\documentclass[leqno]{amsart}
\usepackage{amsfonts,amssymb,amsmath,amsgen,amsthm}
\usepackage{hyperref}
\usepackage{color}
\usepackage{bbm}
\newcommand{\msc}[2][2000]{%
  \let\@oldtitle\@title%
  \gdef\@title{\@oldtitle\footnotetext{#1 \emph{Mathematics subject
        classification.} #2}}%
}

\theoremstyle{plain}
\newtheorem{theorem}{Theorem}[section]

\newtheorem{lemma}[theorem]{Lemma}

\newtheorem{proposition}[theorem]{Proposition}

\theoremstyle{remark}
\newtheorem{remark}[theorem]{Remark}

\def\R{{\mathbb R}}
\def\N{{\mathbb N}}
\def\Z{{\mathbb Z}}
\def\T{{\mathbb T}}
\def\Sch{{\mathcal S}}
\def\O{\mathcal O}
\def\F{\mathcal F}

\def\({\left(}
\def\){\right)}
\def\<{\left\langle}
\def\>{\right\rangle}
\def\le{\leqslant}
\def\ge{\geqslant}

\def\1{\mathbbm{1}}

\def\Eq#1#2{\mathop{\sim}\limits_{#1\rightarrow#2}}
\def\Tend#1#2{\mathop{\longrightarrow}\limits_{#1\rightarrow#2}}

\def\d{{\partial}}
\def\eps{\varepsilon}

\def\si{{\sigma}}

\numberwithin{equation}{section}

\begin{document}

\title[Norm inflation for KdV and KP]{Norm inflation and loss of
  regularity in negative
  Sobolev spaces for KdV and KP}  

\author[R. Carles]{R\'emi Carles}
\address{CNRS\\ IRMAR, UMR 6625\\ 35000 Rennes\\ France}
\email{Remi.Carles@math.cnrs.fr}
\begin{abstract}
  We prove norm inflation phenomena for KdV and KP equations in
  negative order Sobolev spaces, in the periodic case, as well as on
  the whole space, on an arbitrarily large scale of
  negative order Sobolev spaces as target spaces. The proof relies on WKB
  analysis for 
  a semiclassical version of the equation, in a weakly nonlinear
  r\'egime, and the creation of the zero Fourier mode by resonant
  interaction. Unlike in previous similar results, this average mode has
  a smaller order of magnitude than the initial data, which requires a
  more detailed WKB analysis. 
\end{abstract}
\maketitle

\section{Introduction}
\label{sec:intro}

\subsection{Setting}

We consider the Korteweg-de Vries (KdV) equation 
\begin{equation}
  \label{eq:kdv}
   \d_t v +\d_x^3 v =6v\d_x v,\quad v_{\mid t=0}=v_0,
 \end{equation}
 either on the line, $x\in \R$, or on the circle, $x\in \T$, and
the Kadomtsev-Petviashvili (KP) equation, 
\begin{equation}
  \label{eq:kp}
  \d_t v +v\d_x v +\d_x^3 v +\lambda \d_x^{-1}\d_y^2 v =0,\quad
  v_{\mid t=0}=v_0, 
\end{equation}
with $\lambda=\pm 1$, either on $\R^2$ or $\T^2$. The case
$\lambda=+1$ corresponds to the ``defocusing'' KP-II equation, while
the case $\lambda=-1$ corresponds to the ``focusing'' KP-I equation.
We refer to \cite{KleinSaut21} for an extensive
bibliography regarding 
the origin and the mathematical analysis of these models. As the sign of
$\lambda$ plays no role in our analysis, we will evoke the KP equation
to address both KP-I and KP-II. 
For both KdV and KP, we prove norm inflation phenomena in negative order
Sobolev spaces, with arbitrary loss of regularity. In the sequel,
$M^d$ denotes either $\R^d$ or $\T^d$, 
and when $d=1$, we simply write $M$. 
\smallbreak

 On $\R$, the KdV equation is invariant under
 the scaling
 \begin{equation*}
   v(t,x)\to
  \Lambda^2v\(\Lambda^3 t,\Lambda x\),\quad \Lambda>0,
 \end{equation*}
 which leave the Sobolev norm $H^{-3/2}(\R)$ invariant.
 In \cite{KillipVisan2019}, Killip and Vi\c san proved that the KdV flow map could be
 uniquely, continuously extended to a jointly continuous map $\R\times
 H^{-1}(M)\to H^{-1}(M)$. This result is sharp in the sense that it is
 not possible to consider a similar statement for a weaker Sobolev
 regularity, as proven by Luc Molinet. In \cite{Molinet2011} for $x\in
 \R$, \cite{Molinet2012} 
for $x\in\T$, he proved that for any $s<-1$ and any $T>0$, the flow
map cannot be extended from $H^{s}(M)$ to $\mathcal D'(]0,T[\times M)$.
On the one hand, our statement below is weaker, but on the other hand,
it provides a more quantitative statement, whose proof, unlike in
\cite{Molinet2011,Molinet2012}, does not rely on the Miura transform;
see Remark~\ref{rem:mKdV} below for a more precise discussion
in this direction. Note however that the statement in
\cite{KillipVisan2019} must be read carefully, as it was proven in
\cite{CCT} that on $\R$, for any $-1\le s<-3/4$, the solution operator
fails to be uniformly 
continuous with respect to the $H^s$ norm, and a similar statement
holds on $\T$, for $-2<s<-1/2$. Here again, the proof relies on the
study of the modified KdV equation and the Miura transform.
\bigbreak

While for the Benjamin-Ono equation, analogous results to the ones
recalled above by Molinet are available (in \cite{BiLi01} for $M=\R$
and in \cite{GKT23} for $M=\T$), it seems that as far as KP equations
are concerned, less is known on the lack of well-posedness (we refer
to \cite{KleinSaut21} for a presentation of well-posedness
results). On $\R^2$, the KP equation is invariant under the scaling 
\begin{equation*}
  v(t,x,y)\to
  \Lambda^2v\(\Lambda^3 t,\Lambda x,\Lambda^2y\),\quad \Lambda>0,
\end{equation*}
which leaves the $\dot H^{s_1,s_2}(\R^2)$-norm invariant for
$s_1+2s_2=-\frac{1}{2}$, where
\begin{equation*}
  H^{s_1,s_2}(M^2)=\{ \phi\in \Sch'(M^2)\ ; \
  \|\phi\|_{H^{s_1,s_2}}:=
  \|(1-\d_x^2)^{s_1/2}(1-\d_y^2)^{s_2/2}\phi\|_{L^2(M^2)}<\infty\}, 
\end{equation*}
with the standard modification for the homogeneous space $\dot
H^{s_1,s_2}$. We refer to \cite{KleinSaut21} for a description of the
results on KP equations, and present a short summary here, containing
more recent developments. The sign of $\lambda$ in \eqref{eq:kp} turns
out to be crucial in the nature of the equation: the ``defocusing''
KP-II equation ($\lambda=+1$) turns out to be semilinear, and can be
solved by a fixed point argument. The first breakthrough is due to
Bourgain \cite{Bourgain1993}, who proved global well-posedness for
initial data in $L^2(M)$, both for $M=\R^2$ and $M=\T^2$.
The lowest regularity has been reached so far by Hadac
\cite{Hadac2008}, who proved local well-posedness in $H^{s_1,s_2}(\R^2)$
for any $s_1>-1/2$, $s_2\ge 0$, and then in \cite{HHK09}, where the
limiting case $s_1=-1/2$ was reached, and local well-posedness was
proven in  the homogeneous space $\dot H^{-1/2,0}(\R^2)$. In the
periodic setting $(x,y)\in \T^2$, the best low regularity result seems
to be the one from Herr, Schippa and Tzvetkov \cite{HST26}, who proved
local well-posedness in $H^{s,0}(\T^2)$ for $s>-1/90$. 
\smallbreak

On the other
hand, the ``focusing'' KP-I equation is quasilinear: 
in
\cite{MoSaTz02}, it is shown 
that for any $s_1,s_2\in\R$, the flow
map fails to be $C^2$ from 
$H^{s_1,s_2}(\R^2)$ to $H^{s_1,s_2}(\R^2)$, and in \cite{KoTz08}, it is
proven that the flow map cannot be uniformly continuous in the energy
space.  Well-posedness results then rely on compactness methods. After
the first important result from \cite{MST04}, Ionescu, Kenig and
Tataru \cite{IKT08} have proven global well-posedness in the energy
space
\begin{equation*}
  Y= \{u\in L^2(\R^2),\ \d_x u , \d_x{^-1}\d_y u\in L^2(\R^2)\}. 
\end{equation*}
Recently, Guo and Molinet \cite{GuoMolinet2026} proved local
well-posedness for \eqref{eq:kp} (both KP-I and KP-II) in
$H^{s,0}(\R^2)$ for $s>3/4$.

\subsection{Main results}
\label{sec:results}

\begin{theorem}[Norm inflation for KdV]\label{theo:main-kdv}
  Let $s_1,s_2<-1$.
  \begin{itemize}
  \item There exist a sequence of initial data $(v_0^\eps)_{0<\eps\le
      1}$ in $C_0^\infty(M)$ and $t^\eps\to 0$ such that
    \begin{equation*}
      \|v_0^\eps\|_{H^{s_1}(M)} \to 0 \quad\text{as }\eps\to 0,
    \end{equation*}
    while the solution to \eqref{eq:kdv} satisfies
    \begin{equation*}
       \|v^\eps(t^\eps)\|_{H^{s_2}(M)}\to 0,\quad
       \frac{\|v^\eps(t^\eps)\|_{H^{s_2}(M)}}{
         \|v_0^\eps\|_{H^{s_1}(M)}} \to  \infty\quad\text{as }\eps\to 0.
     \end{equation*}
   \item For any $K>-s_2$ and any $\delta>0$, there exist a sequence
     of initial data $(v_0^\eps)_{0<\eps\le 
       1}$ in $C_0^\infty(M)$ and $t^\eps\to 0$ such that
     \[\|v^\eps_0\|_{H^{s_1}(M)}\Tend \eps 0 0,\] 
     while the solution to \eqref{eq:kdv} satisfies
    \begin{equation*}
      \|v^\eps(t^\eps)\|_{H^\si(M)}>\frac{1}{\delta},\quad \forall \si\in
      [-K,s_2]. 
    \end{equation*}
  \item In particular, for any $K>-s_2$ and any $\delta>0$, there
    exist $v_0\in 
    C_0^\infty(M)$ with $\|v_0\|_{H^{s_1}(M)}<\delta$, and
    $0<t<\delta$ such that the solution to \eqref{eq:kdv} satisfies
    \begin{equation*}
      \|v(t)\|_{H^\si(M)}>\frac{1}{\delta},\quad \forall \si\in
      [-K,s_2]. 
    \end{equation*}
  \end{itemize}
\end{theorem}
The case $\si=s_2=s_1$ corresponds to the phenomenon of norm inflation in
$H^{s_1}(M)$, according to the terminology introduced in
\cite{CCT2}. The range for $\si$ shows that the underlying phenomenon
is stronger.

\smallbreak

In the KP case, we prove the analogue of the last two points from
Theorem~\ref{theo:main-kdv}, and leave out the first one, because it
is somehow weaker. 

\begin{theorem}[Norm inflation for KP]\label{theo:main-kp}
  Let $s_1,s_1',s_2,s_2'\le 0$ such that
  \begin{equation*}
    s_1+2s_2<-1,\quad s'_1+2s'_2<-1,
  \end{equation*}
  and let $K>-s_1'-2s_2'$.
  \begin{itemize}
  \item For any $\delta>0$, there exist a sequence
     of initial data $(v_0^\eps)_{0<\eps\le 
       1}$ in $C_0^\infty(M^2)$ and $t^\eps\to 0$ such that
     \[\|v^\eps_0\|_{H^{s_1,s_2}(M^2)}\Tend \eps 0 0,\] 
     while the solution to \eqref{eq:kp} satisfies
    \begin{equation*}
      \|v^\eps(t^\eps)\|_{H^{\si_1,\si_2}(M^2)}>\frac{1}{\delta},\quad
      \forall \si_1,\si_2\le 0\text{ such that } 
      -K\le \si_1+2\si_2\le s_1'+2s'_2. 
    \end{equation*}
  \item In particular, for any $\delta>0$, there exist $v_0\in
    C_0^\infty(M^2)$ with $\|v_0\|_{H^{s_1,s_2}(M^2)}<\delta$, and
    $0<t<\delta$ such that the solution to \eqref{eq:kp} satisfies
    \begin{equation*}
      \|v(t)\|_{H^{\si_1,\si_2}
        (M^2)}>\frac{1}{\delta},\quad
      \forall \si_1,\si_2\le 0\text{ such that } 
      -K\le \si_1+2\si_2\le s_1'+2s'_2. 
    \end{equation*}
  \end{itemize}
\end{theorem}
Like in the KdV case, our result is stronger than merely a norm
inflation in a fixed Sobolev space, but since analogues of the results
by Molinet do not seem to be available in the KP case, all the results
from Theorem~\ref{theo:main-kp} appear to be new. We note however that
unlike in the KdV case, the numerology may not be sharp, in the sense
that our assumption reads $s_1+2s_2<-1$, while scaling suggests that
the upper bound should be $s_1+2s_2<-1/2$. At this stage, we have no
explanation for this discrepancy. We emphasize
that the method of proof, described in the next section, is quite
similar to the one followed in \cite{CaKa17}: there, the nonlinear
Schr\"odinger equation is considered in the periodic setting $x\in \T^d$, $d\ge
1$, and the Sobolev regularity for which initial smallness is
considered is sharp, except in the one-dimensional cubic case, which
is the only known integrable case. Ill-posedness with norm inflation
in this case is established in \cite{OhWang2018} (see also
\cite{Kishimoto2019,Oh2017}), but without the loss
of regularity obtained in \cite{CaKa17}. 
\subsection{Scheme of the proof}
\label{sec:scheme}

We describe the scheme of the proof in the KdV case, the idea being
similar for KP. 
The proof relies on high frequency analysis of the semiclassically
scaled KdV equation,
\begin{equation}
  \label{eq:u-kdv}
  \eps \d_t u^\eps +\eps^3\d_x^3 u^\eps =6 \eps^2 u^\eps\d_x u^\eps,
\end{equation}
in the limit $\eps\to 0$, with initial data of the form (at least as a
first approximation)
\begin{equation*}
  u^\eps(0,x) = \alpha_1(x) e^{ix/\eps}+ \alpha_{-1}(x) e^{-ix/\eps}.
\end{equation*}
In the case $M=\T=\R/2\pi\Z$, to guarantee the periodicity of
$u^\eps(0,\cdot)$, we assume that the parameter $\eps$ is of the form
$\eps=1/N$ for some $N\in \N$. 
The presence of rapid oscillations implies that in the limit $\eps\to
0$, we have the order of magnitude
\begin{equation*}
  \|u^\eps(0)\|_{H^s(M)}\approx \eps^{-s}, \quad \forall s\in \R.
\end{equation*}
In particular, for negative $s$, we consider small data in $H^s$. 
In WKB analysis (or geometric optics approximation, see
e.g. \cite{RaBook2}), one seeks an ansatz of the form 
\begin{equation}\label{eq:OG}
  u^\eps(t,x) \approx \sum_j \(a_j(t,x)+\eps
  b_j(t,x)+\eps^2c_j(t,x)+\dots\) e^{i\phi_j(t,x)/\eps}, 
\end{equation}
where we keep the presentation on a formal level in this
subsection. Plugging this formula into \eqref{eq:u-kdv}, we first
solve the $\O(\eps^0)$ equations,
\begin{equation*}
  \d_t \phi_j-(\d_x \phi_j)^3=0.
\end{equation*}
For plane wave oscillations at initial time,
\begin{equation*}
  \phi_j(0,x) =jx,
\end{equation*}
the solution to the above eikonal equation is given by the dispersion
relation
\begin{equation*}
  \omega(j) =j^3,\quad\text{hence}\quad \phi_j(t,x)=jx+j^3 t.
\end{equation*}
Due to the factor $\eps^2$ in front of the nonlinearity in
\eqref{eq:u-kdv}, the nonlinearity (possibly) plays some role only at
next order, $\O(\eps^1)$: in terms of geometric optics, this is a
\emph{weakly nonlinear} r\'egime. Still on a formal level, the
nonlinear interaction $\eps^2 u^\eps\d_x u^\eps$ involves products of
exponentials $e^{i\phi_j/\eps}$ from \eqref{eq:OG}. The phase
$\phi_j+\phi_k$ solves the eikonal equation if and only if
$\phi_j+\phi_k=\phi_{j+k}$ (since the space factor is $(j+k)x$), hence
\begin{equation*}
  (j+k)^3=j^3+k^3\Longleftrightarrow 3jk(j+k)=0.
\end{equation*}
If not present initially, we say that the phase $\phi_{j+k}$ is
created by resonant interaction. 
In the case that we consider, $u^\eps_{\mid t=0}$ contains $\phi_1$
and $\phi_{-1}=-\phi_1$: the zero mode is created by resonant
interaction. Unlike the initial data, non-oscillatory terms have $H^s$
norms whose behavior as $\eps\to 0$ is independent of
$s\in \R$: the zero mode may become dominant in negative order Sobolev
spaces, and this is the key mechanism leading to
Theorem~\ref{theo:main-kdv}. 
\smallbreak

Lack of well-posedness in negative order spaces for nonlinear
dispersive equations, due to a transfer from
high to low frequencies, goes back to \cite{BKPSV,KPV01} (focusing
case) and \cite{CCT,CCT2,CCTper} (focusing and defocusing cases).  
The idea that the creation of the zero mode by resonant interaction
  may cause norm inflation on a scale of negative order Sobolev spaces
  (and not only in the same Sobolev space) appears in \cite{CDS12},
  in the case of the  multidimensional nonlinear Schr\"odinger equation
  on $\R^d$ and the (2D) Davey-Stewartson system. Stronger results (in
  terms of Sobolev indices and space dimension)
  were obtained in the periodic case $x\in \T^d$ in
  \cite{CaKa17}. Similar results can be found in \cite{BhCa20}
  (Sobolev spaces are replaced by Fourier-Lebesgue or modulation
  spaces), \cite{BihmaniHaque2022} (wave equation),
  \cite{BihmaniHaque2023} (fractional nonlinear Schr\"odinger
  equation).  In all these cases, the zero mode has the same
  order of magnitude at the initial data. In the present case however,
  we will see that the zero 
  mode comes with a factor $\eps$, that is, for some $\tau>0$
  independent of $\eps$,
  \begin{equation*}
    u^\eps(\tau,x) = \eps a_0(\tau,x)+\text{ oscillatory terms} +
    \text{ remainder}, \quad\text{with}\quad a_0(\tau,\cdot)\not\equiv 0.
  \end{equation*}
  The fact that KdV and KP are integrable (see
  \cite{KleinSaut21}), thus enjoying
  extra cancellations, is likely to explain the difference with most
  of the
  above results for nonlinear 
  Schr\"odinger or wave equations; in
  \cite{CDS12}, the zero mode is created (at leading order) by
  resonant interaction for 
 Davey-Stewartson equations, except in the integrable case. In
 \cite{CaKa17}, the WKB expansion is different for the one-dimensional
 cubic Schr\"odinger equation, which is the only one known to be
 integrable: the zero mode appears with an extra factor
 $\eps$ (like described above), whose presence may be understood by
 considering the 
 semiclassical limit for each Picard's iterate (like in \cite{JMR93}),
 and performing an integration by parts in time in Duhamel's
 integral (nonstationary phase argument).

\begin{remark}[mKdV]\label{rem:mKdV}
  In the case of the semiclassically scaled  modified KdV equation,
 \begin{equation*}
  \eps \d_t u^\eps +\eps^3\d_x^3 u^\eps =\pm \eps^2 \(u^\eps\)^2\d_x u^\eps,
\end{equation*}
our approach seems inconclusive. Indeed, it is impossible to create
the zero mode by resonant interaction. We must now consider the
interaction of three phases, $\phi_{k_j}$, $j=1,2,3$: the phase
$\phi_{k_1}+\phi_{k_2}+\phi_{k_3}=0$  if and
only if
\begin{equation*}
  k_1+k_2+k_3=0,\quad k_1^3+k_2^3+k_3^3=0.
\end{equation*}
Plugging the identity $k_3=-(k_1+k_2)$ into the second formula yields
\begin{equation*}
  -3k_1k_2k_3=3k_1k_2(k_1+k_2)=0.
\end{equation*}
Therefore, the zero mode cannot be \emph{created} by this mechanism:
at least one of the $\phi_{k_j}$'s must be zero. This shows that the
present approach is different from the one in
\cite{Molinet2011,Molinet2012} since there, the author first considers
the modified KdV equation, in order to infer results for the KdV
thanks to the Miura transform. More precisely, the fact that the zero
mode cannot be created by resonant interaction of nonzero modes
suggests that the zero Fourier mode (in space) may appear, but with an
extra power of $\eps$ (like for the one-dimensional cubic
Schr\"odinger equation considered in \cite{CaKa17}), which would
yield, with initial data essentially of the same form as above,
 \begin{equation*}
    u^\eps(\tau,x) = \eps ^2a_0(\tau,x)+\text{ oscillatory terms} +
    \text{ remainder}, \quad\text{with}\quad a_0(\tau,\cdot)\not\equiv 0.
  \end{equation*}
Using scaling arguments like presented below does not make it possible
to infer a norm inflation result then, as the zero mode is too small
here. 
\end{remark}

\subsection{Organization of the paper}
\label{sec:contents}

In Section~\ref{sec:kdv}, we give details of the proof of
Theorem~\ref{theo:main-kdv}, up to the derivation of the WKB
approximation. In Section~\ref{sec:kp}, we proceed similarly in the KP
case. The construction of the WKB ansatz is given in appendices.
Appendix~\ref{sec:NL} contains a general computation regarding the
action of the Burgers nonlinearity on WKB type
functions. In Appendix~\ref{sec:BKW-KdV}, we present the construction of
WKB ansatz for the semiclassical KdV equation,
Appendix~\ref{sec:BKW-KP} provides the analogue result in the KP case.

\section{KdV: proof of Theorem~\ref{theo:main-kdv}}
\label{sec:kdv}

\subsection{Scaling}
\label{sec:scaling}

Let $v$ solve \eqref{eq:kdv} and consider
\begin{equation*}
  u^\eps(t,x) = \eps^\alpha v\(\eps^\beta t ,\eps^\gamma x\),
\end{equation*}
for some parameters $\alpha,\beta,\gamma$. This function solves the
semiclassical KdV equation \eqref{eq:u-kdv} if and only if
\begin{equation*}
  1+\beta=3+3\gamma=2+\alpha+\gamma.
\end{equation*}
Keeping $\beta$ as the only parameter yields
\begin{equation*}
  \alpha=\frac{2\beta-1}{3},\quad \gamma=\frac{\beta-2}{3}. 
\end{equation*}
As we start from 
\begin{equation*}
  u^\eps(0,x) = \alpha_1(x) e^{ix/\eps}+ \alpha_{-1}(x) e^{-ix/\eps},
\end{equation*}
this gives in terms of $v$:
\begin{equation}\label{eq:CI-v}
  v_0^\eps(x)=
  \eps^{\frac{1-2\beta}{3}}\alpha_1\(x\eps^{\frac{2-\beta}{3}}\)
  e^{ix/\eps^{\frac{1+\beta}{3}}} +
  \eps^{\frac{1-2\beta}{3}}\alpha_{-1}\(x\eps^{\frac{2-\beta}{3}}\) 
  e^{-ix/\eps^{\frac{1+\beta}{3}}}.
\end{equation}
For $\beta=2$, it is straightforward to estimate $v_0^\eps$ in
$H^s$. In the first point of Theorem~\ref{theo:main-kdv}, we also consider
$\beta<2$, and $v_0^\eps$ is measured  in
$H^s$ thanks to the next subsection.
\subsection{Estimating oscillatory terms}
The following result is a direct adaptation of \cite[Lemma~5.1]{CDS12}:
\begin{lemma}\label{lem:5.1}
  Let $0<\beta\le 2$. For $f\in \Sch(\R)$ and $\kappa\in \R$, 
  we denote
  \begin{equation*}
    I^\eps(f,\kappa)(x)= f\(x
  \eps^{(2-\beta)/3}\) e^{i\kappa x/\eps^{(1+\beta)/3}}.
  \end{equation*}
$(1)$ Let $\kappa\not=0$. For all $s\le 0$,
there exists $C=C(\sigma,\kappa)$ such that for all $f\in\Sch(\R^d)$,   
\begin{equation*}
   \|I^\eps(f,\kappa)\|_{H^s(\R)}^2 \le
   C\eps^{\frac{\beta-2}{3}+2|s|\frac{\beta+1}{3}}
\|f\|^2_{H^{|s|}(\R)}.
\end{equation*}
$(2)$ If $s\le 0$ and $\beta<2$,
\begin{equation*}
  \|I^\eps(f,0)\|_{H^s(\R)}^2
  =\eps^{(\beta-2)/3}\(\|f\|_{L^2(\R)}^2+o(1)\)\quad\text{as }\eps\to
  0. 
\end{equation*}
$3$ If $s\le 0$ and $\beta=2$, $ \|I^\eps(f,0)\|_{H^s(\R)}^2
  =\|f\|_{H^s(\R)}^2$. 
\end{lemma}

\begin{proof}
    We compute
\begin{align*}
  \widehat{I^\eps(f,\kappa)}(\xi) &= 
  \frac{1}{\sqrt{2\pi}}\int 
  e^{-ix \xi} f\(x
  \eps^{(2-\beta)/3}\) e^{i\kappa  x/\eps^{(1+\beta)/3}}dx \\
&=
  \eps^{(\beta-2)/3} \frac{1}{\sqrt{2\pi}}\int 
  e^{-iy \xi/\eps^{(2-\beta)/3} } f\(y\) e^{i\kappa
    y/\eps}dy  \\
&= \eps^{(\beta-2)/3} \widehat f\(
\frac{\xi}{\eps^{(2-\beta)/3}}- \frac{\kappa}{\eps}\). 
\end{align*}
Therefore, 
\begin{align*}
  \|I^\eps(f,\kappa)\|_{H^s(\R)}^2 & = \int \<\xi\>^{2s} \left\lvert
 \widehat{I^\eps(f,\kappa)} (\xi)\right\rvert^2d\xi \\ 
& = \eps^{2\frac{\beta-2}{3}}\int \<\xi\>^{2s} \left\lvert
\widehat f\(
\frac{\xi}{\eps^{(2-\beta)/3}}- \frac{\kappa}{\eps}\)\right\rvert^2d\xi.
\end{align*}
To prove the first point, we write, for $s\le 0$, and $\beta\le 2$,
\begin{align*}
&  \|I^\eps(f,\kappa)\|_{H^s(\R^d)}^2= \\
&= \eps^{2\frac{\beta-2}{3}}\int \<\xi\>^{2s} 
\<\frac{\xi}{\eps^{(2-\beta)/3}}- \frac{\kappa}{\eps}\>^{2s}
\<\frac{\xi}{\eps^{(2-\beta)/3}}- \frac{\kappa}{\eps}\>^{2|s|}
\left\lvert 
\widehat f\(
\frac{\xi}{\eps^{(2-\beta)/3}}-
\frac{\kappa}{\eps}\)\right\rvert^2d\xi\\
&\le \eps^{\frac{\beta-2}{3}}\sup_{\xi \in \R}\(\<\xi\>^{-1}
\<\frac{\xi}{\eps^{(2-\beta)/3}}-
\frac{\kappa}{\eps}\>^{-1}\)^{2|s|}\|f\|^2_{H^{|s|}(\R)} .
\end{align*}
Next, write 
\begin{align*}
\<\frac{\xi}{\eps^{(2-\beta)/3}}-\frac{\kappa}{\eps}\>^{-1} 
= \< \eps^{(\beta-2)/3} 
\left( \xi - \frac{\kappa}{\eps^{(1+\beta)/3}} \right) \>^{-1} 
 \le \< \xi - \frac{\kappa}{\eps^{(1+\beta)/3}} \>^{-1},
\end{align*}
where we have used the assumption $\beta\le 2$ and the property
$0<\eps\le 1$. 
Peetre inequality (see e.g. \cite{Alazard2024}) yields, for
$\kappa\not =0$, 
\begin{equation*}
  \<\xi\>^{-1}\< \xi - \frac{\kappa}{\eps^{(1+\beta)/3}} \>^{-1}\le
  \<\frac{\kappa}{\eps^{(1+\beta)/3}} \>^{-1} \lesssim \eps^{(1+\beta)/3},
\end{equation*}
hence the first point of the lemma. To prove the second point,
write 
\begin{align*}
 \|I^\eps(f,0)\|_{H^s(\R)}^2 & =  \eps^{2\frac{\beta-2}{3}}
\int \<\xi\>^{2s} \left\lvert 
\widehat f\(
\frac{\xi}{\eps^{(2-\beta)/3}}\)\right\rvert^2d\xi\\
&=\eps^{\frac{\beta-2}{3}}\int \<\eps^{(2-\beta)/3}\xi\>^{-2|s|} \left\lvert 
\widehat f\(
\xi\)\right\rvert^2d\xi.
\end{align*}
In the case $\beta<2$, we conclude thanks to the Dominated Convergence
Theorem. 
The third point of the lemma is obvious.  
\end{proof}

\subsection{Small data?}

Applying Lemma~\ref{lem:5.1} to \eqref{eq:CI-v}, we find
for $s\le 0$ and $0<\beta\le 2$:
\begin{equation*}
  \|v_0^\eps\|_{H^s}\lesssim
  \eps^{\frac{1-2\beta}{3}+\frac{\beta-2}{6}+|s|\frac{\beta+1}{3}}=
  \eps^{|s|\frac{\beta+1}{3}-\frac{\beta}{2}}. 
\end{equation*}
The power of $\eps $ is positive if and only if
\begin{equation*}
  |s|>\frac{3\beta}{2(\beta+1)}. 
\end{equation*}
The geometric optics approximation will show that the zero mode is
created by resonant interaction, through a term of size
$\eps$. Suppose that this term is nontrivial at time $\tau$ in
$u^\eps$, then for $\si<0$,
\begin{equation*}
  \|v(\eps^\beta\tau)\|_{H^\si}\approx
  \underbrace{\eps^{\frac{1-2\beta}{3}}}_{\text{scaling factor }\eps^{-\alpha}}\times
  \underbrace{\eps}_{\text{size of the zero mode}}=\eps^{\frac{2}{3}(2-\beta)}. 
\end{equation*}
For the second case of Theorem~\ref{theo:main-kdv}, we will consider
$\beta=2$, and pay more precise attention to the actual form of the
zero mode. For the first case, we will let $\beta=2-\eta$ with
$0<\eta\ll 1$: we have
\begin{equation*}
  \frac{\|v(\eps^\beta\tau)\|_{H^\si}}{\|v_0^\eps\|_{H^s}}\Tend \eps 0
  \infty \quad \Longleftrightarrow\quad
  1+\frac{\eta}{6}<|s|\(1-\frac{\eta}{3}\). 
\end{equation*}
If we assume $s<-1$, we can always find $\eta>0$ such that the above
inequality holds. Therefore, the proof of Theorem~\ref{theo:main-kdv} now
boils down to the justification of geometric optics approximation, in
the sense that we justify the presence and role of the zero mode.

\subsection{Functional setting}
\label{sec:Hseps}

We resume classical methods and notations from geometric optics, which
can be found in e.g.
\cite[Chapter~8]{RaBook2}.  
We consider the $\eps$-dependent norm
\begin{equation}\label{eq:H2eps}
  \|f\|_{H^2_\eps}^2 = \|f\|_{L^2}^2 + \|\eps^2\d_x^2 f\|_{L^2}^2.
\end{equation}
Introduce the scaling
\begin{equation*}
  g(x) = f(\eps x).
\end{equation*}
Sobolev embedding $\|g\|_{L^\infty}\lesssim \|g\|_{H^1}$ implies
$\|\d_x g\|_{L^\infty}\lesssim \|g\|_{H^2}$, which leads to
\begin{equation}\label{eq:sobolev}
  \|\eps\d_x f\|_{L^\infty}
  \lesssim \eps^{-1/2}\|f\|_{H^2_\eps},
\end{equation}
where the implicit constant is independent of $\eps$.
\subsection{The approximation}

In Appendix~\ref{sec:BKW-KdV}, we prove:
\begin{proposition}\label{prop:approx-kdv}
  Let $\alpha_1\in C_0^\infty(M)$. We can find smooth functions
  $a_j,b_j$ and $c_j$ such that
  $u^\eps_{\rm app} $, defined by
\begin{align*}
  u_{\rm app}^\eps(t,x)
  &= \(a_1(t,x) +\eps b_1(t,x)+\eps^2c_1(t,x)\)
    e^{i\phi_1(t,x)/\eps} \\ 
 &\quad +\eps \( a_2(t,x)+\eps
    b_2(t,x)\)e^{2i\phi_1(t,x)/\eps}
  +\eps^{2} a_3(t,x)e^{3i\phi_1(t,x)/\eps}\\
 &\quad +{\rm c.c.}+\eps a_0(t,x),
  \end{align*}
where we recall that $\phi_1$ is given by $\phi_1(t,x)=x+t$,
solves
 \begin{equation*}
  \eps\d_tu_{\rm app}^\eps  +\eps^3\d_x^3u_{\rm app}^\eps =6
  \eps^{2}u_{\rm app}^\eps \d_x u_{\rm app}^\eps + \Sigma^\eps,
\end{equation*}
where,  for all $T>0$,
\begin{equation*}
  \|\Sigma^\eps\|_{L^\infty([0,T],H^2_\eps)}=\O\(\eps^3\). 
\end{equation*}
We have explicitly
\begin{equation*}
    a_0(t,x)  = 2\(
    |\alpha_1(x+3t)|^2-|\alpha_1(x)|^2\).
  \end{equation*}
\end{proposition}

We now prove that this result implies:
\begin{proposition}\label{prop:OG}
  Let $\alpha_1\in C_0^\infty(M)$, and $u^\eps$ solve \eqref{eq:u-kdv}
  with
  \begin{equation*}
    u^\eps(0,x) = \alpha_1(x)e^{ix/\eps}
    -\eps\alpha_1^2(x)e^{2ix/\eps}+{\rm c.c.} 
  \end{equation*}
  Then for any $T>0$,
  \begin{align*}
    u^\eps(t,x) &= \(\alpha_1(x+3t)+\eps b_1(t,x)\) e^{i(t+x)/\eps} + \eps
                  a_0(t,x) - \eps \alpha_1(x+3t)^2 e^{2i(t+x)/\eps}\\
    &\quad+{\rm c.c.}
    +\O(\eps^2), 
  \end{align*}
  where $a_0$ and $b_1$ are given by Proposition~\ref{prop:approx-kdv}
  (see \eqref{eq:b1} for $b_1$),
   and the $\O(\eps^2)$ is measured in $L^\infty(0,T;H^2_\eps)$.
 \end{proposition}

\subsection{Proof of Proposition~\ref{prop:OG}}\label{sec:preuve-OG-kdv}
Let $w^\eps=u^\eps - u_{\rm app}^\eps $ denote the error:
\begin{equation*}
   \eps\d_tw^\eps  +\eps^3\d_x^3w^\eps =6
  \eps^{2}\(u^\eps\d_x u^\eps-u_{\rm app}^\eps \d_x u_{\rm app}^\eps
  \)- \Sigma^\eps .
\end{equation*}
Decompose
\begin{equation*}
  u^\eps\d_x u^\eps-u_{\rm app}^\eps \d_x u_{\rm app}^\eps= u^\eps\d_x
  w^\eps +w^\eps \d_x u_{\rm app}^\eps,
\end{equation*}
so that
\begin{equation}
  \label{eq:w}
   \eps\d_tw^\eps  +\eps^3\d_x^3w^\eps =6
  \eps^{2}\(u^\eps\d_x
  w^\eps +w^\eps \d_x u_{\rm app}^\eps\)-\Sigma^\eps .
\end{equation}
We first prove energy estimates in $L^2$ and $H^2$, before using a
bootstrap argument.
\smallbreak

Multiply \eqref{eq:w} by $w^\eps$ and integrate in space:
\begin{align*}
  \frac{\eps}{2}\frac{d}{dt}\|w^\eps\|_{L^2}^2
  &= 6 \eps^2 \int
  w^\eps \( u^\eps\d_x
    w^\eps +w^\eps \d_x u_{\rm app}^\eps\)-\int w^\eps\Sigma^\eps\\
&=- 3 \eps^2 \int
 \d_x u^\eps
    \(w^\eps\)^2 +6 \eps^2 \int \(w^\eps\)^2 \d_x u_{\rm
  app}^\eps-\int w^\eps\Sigma^\eps 
  . 
\end{align*}
We infer
\begin{equation*}
 \frac{\eps}{2}\frac{d}{dt}\|w^\eps\|_{L^2}^2\lesssim \eps
 \(\|\eps\d_x u^\eps\|_{L^\infty}+ \|\eps\d_x u_{\rm
   app}^\eps\|_{L^\infty}\) \|w^\eps\|_{L^2}^2 +
 \eps^3 \|w^\eps\|_{L^2},
\end{equation*}
hence
\begin{equation*}
  \|w^\eps(t)\|_{L^2}\le \|w^\eps(0)\|_{L^2}+C\int_0^t
  \(1+\|\eps\d_x w^\eps(s)\|_{L^\infty}\) \|w^\eps(s)\|_{L^2}ds +C\eps^2,
\end{equation*}
where we have used
\begin{equation*}
  \|\eps\d_x u_{\rm
   app}^\eps\|_{L^\infty}\lesssim 1.
\end{equation*}

We now pass to the energy estimate in $H^2$. 
Apply the operator $\eps^2\d_x^2$ to \eqref{eq:w}:
\begin{equation*}
  \(\eps\d_t  +\eps^3\d_x^3\)\eps^2\d_x^2 w^\eps =6
  \eps^{4}\d_x^2\(u^\eps\d_x
  w^\eps +w^\eps \d_x u_{\rm app}^\eps\)-\eps^2\d_x^2\Sigma^\eps .
\end{equation*}
Multiply by $\eps^2\d_x^2 w^\eps$ and integrate in space:
\begin{equation*}
  \frac{\eps}{2}\frac{d}{dt}\|\eps^2\d_x^2 w^\eps\|_{L^2}^2 = 6
  \eps^2\int \eps^2 \d_x^2 w^\eps  \eps^{2}\d_x^2\(u^\eps\d_x
  w^\eps +w^\eps \d_x u_{\rm app}^\eps\)- \int \eps^2 \d_x^2
  w^\eps\eps^2\d_x^2\Sigma^\eps. 
\end{equation*}
The last term is obviously controlled by
\begin{equation*}
  \|\eps^2\d_x^2 w^\eps\|_{L^2} \|\eps^2\d_x^2\Sigma^\eps\|_{L^2} \le
  \|\eps^2\d_x^2 w^\eps\|_{L^2} \|\Sigma^\eps\|_{H^2_\eps} \lesssim
  \eps^3 \|\eps^2\d_x^2 w^\eps\|_{L^2} . 
\end{equation*}
For the first part of the right hand side, we introduce the commutator
\begin{equation*}
  \eps\int \eps^2 \d_x^2 w^\eps  \eps^{2}\d_x^2\(u^\eps\d_x
  w^\eps \) =  \eps\int \eps^2 \d_x^2 w^\eps  \times u^\eps\times \eps^{2}\d_x^3\
  w^\eps  +  \int \eps^2 \d_x^2 w^\eps\left[
    \eps^2\d_x^2,u^\eps\right]\eps\d_x  w^\eps. 
\end{equation*}
Integrating by parts,
\begin{equation*}
  \left|\eps\int \eps^2 \d_x^2 w^\eps  \times u^\eps\times \eps^{2}\d_x^3\
  w^\eps \right|=\frac{1}{2} \left|\int \(\eps^2 \d_x^2 w^\eps\)^2 \eps\d_x
  u^\eps\right|\lesssim \(1+\|\eps \d_x w^\eps\|_{L^\infty}\)
  \|\eps^2\d_x w^\eps||_{L^2}^2. 
\end{equation*}
On the other hand,
\begin{equation*}
  \left\| \left[ \eps^2\d_x^2,u^\eps\right] \eps\d_x
  w^\eps\right\|_{L^2}
  \lesssim \|\eps^2\d_x^2 w^\eps \eps\d_x u^\eps\|_{L^2} + \|\eps\d_x
    w^\eps\eps^2\d_x^2 u^\eps\|_{L^2}.
\end{equation*}
The first term is controlled by
\begin{equation*}
  \|\eps\d_x u^\eps\|_{L^\infty} \|\eps^2\d_x^2
    w^\eps\|_{L^2}\lesssim \(1+ \|\eps\d_x w^\eps\|_{L^\infty} \)
    \|\eps^2\d_x^2  w^\eps\|_{L^2}. 
\end{equation*}
We may also write
\begin{align*}
  \|\eps\d_x
    w^\eps\eps^2\d_x^2 u^\eps\|_{L^2}
  &\lesssim \|\eps\d_x w^\eps\|_{L^\infty} \|\eps^2\d_x^2
    w^\eps\|_{L^2}+ \|\eps \d_x w^\eps\eps^2\d_x^2
    u^\eps_{\rm app}\|_{L^2}\\
  &\lesssim \|\eps\d_x w^\eps\|_{L^\infty}  \|\eps^2\d_x^2
    w^\eps\|_{L^2}+ \|\eps\d_x w^\eps\|_{L^2},
\end{align*}
where we have used the bound $\|\eps^2\d_x^2
    u^\eps_{\rm app}\|_{L^\infty}\lesssim 1$. 
Summing the previous estimates, 
\begin{align*}
  \|w^\eps(t)\|_{H^2_\eps}\le \|w^\eps(0)\|_{H^2_\eps} + C\int_0^t
  \(1+\|\eps\d_x w^\eps(s)\|_{L^\infty}\) \|w^\eps(s)\|_{H^2_\eps}ds
  +C\eps^2.  
\end{align*}
We use a bootstrap argument: as long as 
\begin{equation}\label{eq:bootstrap}
  \|\eps\d_x w^\eps(t)\|_{L^\infty}\le 1, 
\end{equation}
which is true at $t=0$, hence on $[0,t^\eps]$ for some $t^\eps>0$ by
continuity, Gr\"onwall lemma yields
\begin{equation*}
   \|w^\eps(t)\|_{H^2_\eps}\le \|w^\eps(0)\|_{H^2_\eps}e^{Ct} +
   C\eps^2 e^{Ct}. 
 \end{equation*}
 By construction, $\|w^\eps(0)\|_{H^2_\eps}=\O(\eps^2)$: 
 as long as  \eqref{eq:bootstrap} holds,
 \begin{equation}\label{eq:erreur}
 \sup_{0\le t\le T}  \|w^\eps(t)\|_{H^2_\eps}\le
   C(T)\eps^2. 
 \end{equation}
 Estimate \eqref{eq:sobolev} implies that
 for all $T>0$, there exists  $\eps(T)>0$ and $C(T)$ 
 such that for $0<\eps\le \eps_T$, \eqref{eq:erreur} is true.

 \subsection{Remaining arguments for the proof of
   Theorem~\ref{theo:main-kdv}}
\label{sec:end}

We give more details regarding the proof of the second assertion of
Theorem~\ref{theo:main-kdv}, where we consider $\beta=2$. Let $\alpha\in
C_0^\infty(M)$ be any 
nontrivial function, such that in addition $\operatorname{supp}
\alpha\subsetneq \T$ in the case $M=\T$. Fix $\tau>0$ such that
\begin{equation*}
  \operatorname{supp}\alpha\cap  \operatorname{supp}\alpha (\cdot
  +3\tau)=\emptyset. 
\end{equation*}
Fix $-K<s_2<-1$ like in Theorem~\ref{theo:main-kdv}:
\begin{equation*}
  \inf_{-K\le \si\le s_2} \||\alpha|^2\|_{H^\si(M)}>0,
\end{equation*}
hence,
\begin{equation*}
  \inf_{-K\le \si\le s_2}
  \left\||\alpha(\cdot+3\tau)|^2-|\alpha|^2|\right\|_{H^\si(M)}>0. 
\end{equation*}
Fix $\delta>0$: for $\alpha_1=N \alpha$ with a sufficiently large
constant $N>0$,
\begin{equation}\label{eq:a0unif}
  \inf_{-K\le \si\le s_2}\|a_0(\tau)\|_{H^\si(M)} = 
2\inf_{-K\le \si\le s_2}
\left\||\alpha_1(\cdot+3\tau)|^2-
  |\alpha_1|^2|\right\|_{H^\si(M)}>\frac{2}{\delta}.
\end{equation}
With this $\alpha_1$ fixed, let $v_0^\eps(x) =
\eps^{-1}\alpha_1(x)e^{ix/\eps} -\alpha_1(x)^2
e^{2ix/\eps}+\text{c.c.}$, that is \eqref{eq:CI-v} with $\beta=2$, up
to the correction (preparation of the initial data) introduced in
Proposition~\ref{prop:approx-kdv} (see Appendix~\ref{sec:BKW-KdV}). As
we have seen in Section~\ref{sec:scaling}, for $s_1<-1$,
\begin{equation*}
  \|v_0^\eps\|_{H^{s_1}}\lesssim \eps^{|s_1|-1}.
\end{equation*}
In view of Proposition~\ref{prop:OG}, the difference
$r^\eps = u^\eps - u_{\rm app}^\eps$
is such that
\begin{equation*}
  \|r^\eps(\tau)\|_{L^2}\lesssim \eps^2.
\end{equation*}
On the other hand, invoking Lemma~\ref{lem:5.1} again (like for the
above estimate of $v_0^\eps$), for $\si<-1$,
\begin{equation*}
  \|\eps^{-1}u^\eps_{\rm app}(\tau)-a_0(\tau)\|_{H^\si} \lesssim
  \eps^{|\si|-1}, 
\end{equation*}
hence
\begin{equation*}
  \|v^\eps(\eps^2\tau)- a_0(\tau)\|_{H^\si} \lesssim  \eps^{|\si|-1}
  +\eps^2. 
\end{equation*}
Setting $t^\eps=\eps^2\tau$, the second point of
Theorem~\ref{theo:main-kdv} follows from \eqref{eq:a0unif}. The last
point of   Theorem~\ref{theo:main-kdv} is then direct, by choosing
$\eps>0$ sufficiently small, given $\delta>0$.

\section{KP: proof of Theorem~\ref{theo:main-kp}}
\label{sec:kp}
The scheme of the proof of Theorem~\ref{theo:main-kp} is the same as
for Theorem~\ref{theo:main-kdv}, so we emphasize the main
modifications. 
The scaling we now consider is given by
\begin{equation*}
  v(t,x,y)=\frac{1}{\eps}u^\eps\(\frac{t}{\eps^2},x,y\).
\end{equation*}
The equation satisfied by $u^\eps$ reads
\begin{equation}\label{eq:kp-semi}
 \eps  \d_t u^\eps + \eps^3\d_x^3 u^\eps +\lambda \eps^3\d_x^{-1}\d_y^2 u^\eps
  +\eps^2u^\eps\d_x u^\eps=0.  
\end{equation}
The major feature for KP equation is the presence and the
understanding of the operator $\d_x^{-1}$ (see
e.g. \cite{KleinSaut21}). In the present framework, we note that its
action  on rapid oscillations requires some caution. Indeed,
\begin{equation*}
  (\eps\d_x)^{-1}\(a(x)e^{ik_1x/\eps}\)
  =\F^{-1}\(\frac{i}{\eps\xi}\hat a\(\xi-\frac{k_1}{\eps}\)\),
\end{equation*}
so if $k_1\not =0$, this function need not belong to $L^2(\R)$. We
therefore write the equation for $u^\eps$ as 
\begin{equation}\label{eq:u-compl}
  \eps^2 \d_x \d_t u^\eps + \eps^4\d_x^4 u^\eps +\lambda \eps^4\d_y^2 u^\eps
  +\frac{\eps}{2} \(\eps \d_x\)^2\(u^\eps\)^2=0.
\end{equation}
We emphasize that for $\eps>0$ fixed, we consider smooth solutions,
and so the order of derivatives can be chosen as convenient. 
\subsection{Derivation of the WKB expansion}
\label{sec:OG-KP}

In accordance with the factor $\eps^4$ in front of the $y$-derivative,
we allow some strong anisotropy in the oscillations for $u^\eps$, 
and we assume, at leading order,
\begin{equation*}
  u^\eps(0,x,y) = \alpha_1(x,y)
  e^{ik_1x/\eps+ik_2y/\eps^2}+{\rm c.c.}
\end{equation*}
We keep $k_1,k_2\not =0$  as parameters in our presentation. We just
note that like for KdV,
$k_1,k_2$ and $\eps$ have to be chosen accordingly in the periodic
setting $(x,y)\in \T^2$.
The characteristic phase associated with the initial oscillations,
\begin{equation*}
  \phi^\eps(0,x,y) = k_1\frac{x}{\eps}+k_2\frac{y}{\eps^2},
\end{equation*}
is of the form
\begin{equation}\label{eq:eik-KP}
  \phi^\eps(t,x,y) =
  k_1\frac{x}{\eps}+k_2\frac{y}{\eps^2}+\omega\frac{t}{\eps}, 
\end{equation}
and the WKB hierarchy will show that $\omega$ is given by the usual
dispersion relation for KP equations,
\begin{equation*}
  \omega = k_1^3-\lambda \frac{k_2^2}{k_1}. 
\end{equation*}
Because of the specific form of the oscillations in $x$ and $y$, the
analogue of $H^k_\eps$ is
\begin{equation}\label{eq:Hkeps}
  \|f\|_{H^k_\eps}^2= \|f\|_{L^2}^2 + \|\eps^k \d_x^kf\|_{L^2}^2 +
  \|\eps^{2k} \d_y^kf\|_{L^2}^2, 
\end{equation}
and the inequality leading to \eqref{eq:sobolev} is replaced with
\begin{equation*}
  \|f\|_{L^\infty}\lesssim \eps^{-3/2}\|f\|_{H^k_\eps},
\end{equation*}
provided that $k>d/2=1$. This follows by considering (now
that we have introduced some anisotropy in the definition of
the $H^k_\eps$ norm)
\begin{equation*}
  g(x,y)=f(\eps x,\eps^2 y),
\end{equation*}
along with the Sobolev embedding (where $\eps$ is not involved)
\begin{equation*}
  \|g\|_{L^\infty(\R^2)}\le C(k) \|g\|_{H^k(\R^2)}, \quad k>1.
\end{equation*}
Therefore, the analogue of \eqref{eq:sobolev} reads
\begin{equation*}
 \|\eps\d_xf\|_{L^\infty}\lesssim \eps^{-3/2}\|f\|_{H^k_\eps},\text{
   provided that }k>2.  
\end{equation*}
The general strategy, in order to use standard energy estimates for KP
equations, consists in constructing an approximate solution $u_{\rm
  app}^\eps$ solving \eqref{eq:u-compl} up to some small source term,
which is itself an $x$-derivative,
\begin{equation}\label{eq:uapp-KP}
   \eps^2 \d_x \d_t u_{\rm app}^\eps + \eps^4\d_x^4 u_{\rm app}^\eps
   +\lambda \eps^4\d_y^2 u_{\rm app}^\eps 
  +\frac{\eps}{2} \(\eps \d_x\)^2\(u_{\rm app}^\eps\)^2=\eps\d_x \Sigma^\eps,
\end{equation}
where $\Sigma^\eps$ is small in $H^k_\eps$ for a minimal $k$. The
notion of minimality concerns the smallest integer $k$ such that we may
infer, like in the KdV case,
\begin{equation*}
  \|v^\eps(\eps^2\tau)-a_0(\tau)\|_{H^{\si_1,\si_2}}=o(1),
\end{equation*}
where $a_0$ will appear by a mechanism
similar to the KdV case. 
The requirements are twofold, in the study of $u^\eps$:
\begin{itemize}
\item We need a remainder term which is $o(\eps)$ in $L^2(\R^2)$ to
  make sure that the term $\eps a_0$ is relevant.
\item We need to make sure that $\|\eps \d_x w^\eps\|_{L^\infty}\le 1$
  for $0<\eps\ll 1$ for the bootstrap argument, where
  $w^\eps=u^\eps-u_{\rm app}^\eps$ (we will actually use a stronger
  inequality).  
\end{itemize}
The first condition implies that the expansion defining $u^\eps_{\rm
  app}$ must go up to terms of order
$\eps^2$ at least. The second condition will follow if we have
\begin{equation*}
  \eps^{-3/2}\|w^\eps\|_{H^k_\eps}\ll 1\quad\text{for some }k>2.
\end{equation*}
Therefore, we choose to measure the smallness of $w^\eps$ in $H^3_\eps$,
and expect at least
\begin{equation*}
  \|w^\eps\|_{H^3_\eps}=\O(\eps^2). 
\end{equation*}
Factorizing $\eps\d_x$, and recalling the factor $\eps$ in front of the
time derivative, this will follow from the construction provided that
we have 
\begin{equation*}
\|\Sigma^\eps\|_{H^3_\eps}=\O(\eps^3).   
\end{equation*}
The analogue of Proposition~\ref{prop:approx-kdv}, stemming from
Appendix~\ref{sec:BKW-KP} is:
\begin{proposition}\label{prop:approx-kp}
  Let $\alpha_1\in C_0^\infty(M^2)$. We can find smooth functions
  $a_j,b_j$ and $c_j$ such that
  $u^\eps_{\rm app} $, defined by
\begin{equation*}
  u^\eps_{\rm app} = \eps\d_x\(\sum_{j=1}^3 \eps^{j-1}\(a_j+\eps
  b_j+\eps^2 c_j\) e^{ij\phi^\eps} + {\rm c.c.}\)+ \eps a_0 ,
 \end{equation*}
where $\phi^\eps$ is given by \eqref{eq:eik-KP},
solves \eqref{eq:uapp-KP}
where,  for all $T>0$,
\begin{equation*}
  \|\Sigma^\eps\|_{L^\infty([0,T],H^2_\eps)}=\O\(\eps^3\). 
\end{equation*}
We have explicitly
\begin{equation*}
  a_1(t,x,y) = a_1\( 0,x+\(3k_1^2 +\lambda\frac{k_2^2}{k_1^2}\)t,y\)=
  \alpha_1 \(x+\(3k_1^2 +\lambda\frac{k_2^2}{k_1^2}\)t,y\),
\end{equation*}
and 
\begin{equation*}
    \d_t a_0  = -k_1^2\d_x|a_1|^2,\quad a_{0\mid t=0}=0.
  \end{equation*}
\end{proposition}

\subsection{Justification of the WKB expansion}

Let $w^\eps = u^\eps-u^\eps_{\rm app}$, where we may choose, to
simplify the presentation, to impose $w^\eps=0$ at $t=0$
(well-prepared initial data). We note that this assumption ensures
that $u^\eps_{\mid t=0}$, like $u^\eps_{{\rm app}\mid t=0}$, is the
$x$-derivative of a smooth function (for $\eps>0$ fixed). By
construction, $w^\eps$ solves 
\begin{equation*}
  \eps^2\d_x\d_t w^\eps +\eps^4\d_x^4 w^\eps
  +\lambda\eps^4\d_y^2w^\eps +
  \frac{\eps}{2}(\eps\d_x)^2\(\(u^\eps\)^2 - \(u^\eps_{\rm app}\)^2\)
  =  -\eps\d_x \Sigma^\eps,\quad w^\eps_{\mid t=0} = 0. 
\end{equation*}
We can now apply the operator $(\eps\d_x)^{-1}$, to get
\begin{equation}\label{eq:w-kp}
 \eps\d_t w^\eps +\eps^3\d_x^3 w^\eps
  +\lambda\eps^3\d_x^{-1}\d_y^2w^\eps +
  \frac{\eps^2}{2}\d_x\(\(u^\eps\)^2 - \(u^\eps_{\rm app}\)^2\)
  =  - \Sigma^\eps,\quad w^\eps_{\mid t=0} = 0.  
\end{equation}
We perform energy estimates, using the fact that $\d_x^3$ and
$\d_x^{-1}\d_y^2$ are skew-adjoint (the linear propagator is unitary
on $H^s(M^2)$ for all $s\in \R$). For the energy estimate in $L^2$,
we can resume the computation from
Section~\ref{sec:preuve-OG-kdv}, and get
\begin{equation*}
  \|w^\eps(t)\|_{L^2}\le C\int_0^t
  \(1+\|\eps\d_x w^\eps(s)\|_{L^\infty}\)\|w^\eps(s)\|_{L^2}ds +C\eps^2,
\end{equation*}
since we still have
\begin{equation*}
  \|\eps\d_x u_{\rm
   app}^\eps\|_{L^\infty}\lesssim 1.
\end{equation*}
We then apply the operator $\eps^3 \d_x^3$ to \eqref{eq:w-kp}, multiply by
$\eps^3\d_x^3 w^\eps$ and integrate:
\begin{equation*}
  \frac{\eps}{2}\frac{d}{dt}\|\eps^3\d_x^3 w^\eps\|_{L^2}^2
  =  -\eps^2\int_{M^2}\eps^3\d_x^3 w^\eps \eps^3\d_x^3\(u^\eps \d_x
    w^\eps+w^\eps\d_x u^\eps_{\rm app}\) - \int_{M^2} \eps^3\d_x^3
    w^\eps \eps^3 \d_x^3\Sigma^\eps.
\end{equation*}
The last term is controlled by Cauchy-Schwarz inequality and
Proposition~\ref{prop:approx-kp},
\begin{equation*}
\left|\int_{M^2} \eps^3\d_x^3
    w^\eps \eps^3 \d_x^3\Sigma^\eps  \right|\le \|\eps^3\d_x^3
  w^\eps\|_{L^2}\|\Sigma^\eps\|_{H^3_\eps}. 
\end{equation*}
Introducing the commutator,
\begin{equation*}
  \eps\int \eps^3\d_x^3 w^\eps \eps^3\d_x^3\(u^\eps \d_x
    w^\eps\)=\eps \int \eps^3\d_x^3 w^\eps u^\eps \eps^3\d_x^4
    w^\eps+\int \eps^3\d_x^3 w^\eps\left[ \eps^3\d_x^3,
    u^\eps \right] \eps\d_x  w^\eps, 
  \end{equation*}
  the first term on the right hand side is integrated by parts,
  \begin{equation*}
\left|  \eps \int \eps^3\d_x^3 w^\eps u^\eps \eps^3\d_x^4
    w^\eps\right|=   \frac{1}{2}\left|\int \(\eps^3\d_x^3 w^\eps\)^2
    \eps\d_x 
    u^\eps\right| \lesssim \(1+\|\eps \d_x
  w^\eps\|_{L^\infty}\)\|w^\eps\|_{H^3_\eps}. 
\end{equation*}
The commutator is estimated by
\begin{align*}
  \left\| \left[ \eps^3\d_x^3,
    u^\eps \right] \eps\d_x  w^\eps\right\|_{L^2}
  &\lesssim \|\eps\d_x w^\eps\eps^3\d_x^3
    u^\eps\|_{L^2} + \|\eps^2\d_x^2 w^\eps\eps^2\d_x^2 u^\eps\|_{L^2}\\
&\quad    + \|\eps \d_x u^\eps\eps^3\d_x^3 w^\eps\|_{L^2}. 
\end{align*}
The last term on the right hand side is controlled by
\begin{equation*}
   \|\eps \d_x u^\eps\|_{L^\infty}\|\eps^3\d_x^3
    w^\eps\|_{L^2} \lesssim \(1+\|\eps\d_x w^\eps\|_{L^\infty}\)
    \|w^\eps\|_{H^3_\eps}. 
\end{equation*}
We may write like in the KdV case
\begin{align*}
   \|\eps\d_x w^\eps\eps^3\d_x^3
    u^\eps\|_{L^2}
    &\le
   \| \eps\d_x w^\eps\eps^3\d_x^3   w^\eps\|_{L^2}
      + \| \eps\d_x w^\eps\eps^3\d_x^3   u^\eps_{\rm app}\|_{L^2}\\
    &\le
   \|\eps\d_x w^\eps \|_{L^\infty}\|w^\eps\|_{H^3_\eps} + \|\eps\d_x
   w^\eps\|_{L^2} \underbrace{\|\eps^3\d_x^3   u^\eps_{\rm
      app}\|_{L^\infty}}_{\lesssim 1}\\
   &\lesssim \(1+  \|\eps\d_x w^\eps
     \|_{L^\infty}\)\|w^\eps\|_{H^3_\eps} . 
\end{align*}
Similarly,
\begin{align*}
  \|\eps^2\d_x^2 w^\eps\eps^2\d_x^2 u^\eps\|_{L^2}
  &\le \|\eps^2\d_x^2
  w^\eps\eps^2\d_x^2 w^\eps\|_{L^2} +
    \|\eps^2\d_x^2 w^\eps\eps^2\d_x^2 u^\eps_{\rm app}\|_{L^2}\\
  & \lesssim \|\eps^2\d_x^2 w^\eps\|_{L^4}^2 + \|\eps^2\d_x^2
    w^\eps\|_{L^2}. 
\end{align*}
The $L^4$ norm is estimated by resuming  the scaling transform
\begin{equation*}
  g(x,y) = f(\eps x,\eps^2 y),
\end{equation*}
and invoking the Sobolev embedding $\|g\|_{L^4}\lesssim \|g\|_{H^1}$,
leading to 
\begin{equation*}
  \|\eps^2\d_x^2f \|_{L^4}\lesssim \eps^{-3/4} \|f\|_{H^3_\eps},
\end{equation*}
hence
\begin{equation*}
  \|\eps^2\d_x^2 w^\eps\|_{L^4}^2\lesssim \eps^{-3/2} \|w^\eps\|_{H^3_\eps}^2.
\end{equation*}
\begin{remark}
 One could simplify a step in the above estimate, by writing
 \begin{equation*}
    \int \eps^3\d_x^3 w^\eps \eps^2\d_x^2 w^\eps \eps^2 \d_x^2 u^\eps
    = \int \eps^3\d_x^3 w^\eps \eps^2\d_x^2 w^\eps \eps^2 \d_x^2
    w^\eps
    +\int \eps^3\d_x^3 w^\eps \eps^2\d_x^2 w^\eps \eps^2 \d_x^2
    u^\eps_{\rm app},
  \end{equation*}
  and noticing that the first integrand on the right hand side is an
  exact derivative. However, this argument does not seem to be
  extendable to the case where the operator $\eps^3\d_x^3$ is replaced
  by $\eps^6\d_y^3$, like needed in order to conclude. 
\end{remark}
Finally, we apply $\eps^6 \d_y^3$ to \eqref{eq:w-kp}, multiply by
$\eps^6\d_y^3 w^\eps$ and integrate, to get a similar estimate, so
that summing the three integrated inequalities, we get
\begin{align*}
  \|w^\eps(t)\|_{H^3_\eps}
  &\lesssim
 \int_0^t\|\Sigma^\eps(s)|_{H^3_\eps}ds + \int_0^t  \(1+  \|\eps\d_x w^\eps
    (s) \|_{L^\infty}\)\|w^\eps(s)\|_{H^3_\eps}ds \\
  &\quad + \eps^{-3/2}
    \int_0^t  \|w^\eps(s)\|^2_{H^3_\eps}ds  .
\end{align*}
We strengthen the previous bootstrap argument: for $T>0$ fixed, we
consider the time interval $I^\eps\ni 0$ on which 
\begin{equation*}
    \|w^\eps(t)\|_{H^3_\eps}\le 2
    \|\Sigma^\eps\|_{L^\infty(0,T;H^3_\eps)},\quad t\in I^\eps.
\end{equation*}
As the right hand side is $\O(\eps^2)$, so long as this bootstrap
argument holds,
\begin{equation*}
  \|\eps\d_x w^\eps(t)\|_{L^\infty}\lesssim
  \eps^{-3/2} \|w^\eps(t)\|_{H^3_\eps}\lesssim \sqrt\eps \le 1, 
\end{equation*}
provided that $\eps$ is sufficiently small,
we conclude that, choosing $\eps(T)>0$
sufficiently small, the bootstrap argument is valid for $t\in [0,T]$
provided that $0<\eps\le \eps(T)$, and we infer the analogue of
Proposition~\ref{prop:OG}:
\begin{proposition}\label{prop:OG-kp}
  Let $\alpha_1\in C_0^\infty(M)$, $u^\eps_{\rm app}$ given by
  Proposition~\ref{prop:approx-kp}, and $u^\eps$ solve \eqref{eq:kp-semi}
  with $u^\eps_{\mid t=0} = u^\eps_{{\rm app}\mid t=0} $. 
  Then for any $T>0$,
  \begin{equation*}
    \|u^\eps-u^\eps_{\rm app}\|_{L^\infty(0,T;H^3_\eps)}=\O(\eps^2).
  \end{equation*}
 \end{proposition}
\subsection{Proof of Theorem~\ref{theo:main-kp}}

We note that we may assume that the initial leading order profile,
$\alpha_1$, is of the form
\begin{equation*}
  \alpha_1(x,y) = \alpha_{1,x}(x)\alpha_{1,y}(y),
\end{equation*}
for $\alpha_{1,x},\alpha_{1,y}\in C^\infty_0(M)$. Invoking
Lemma~\ref{lem:5.1} with $\beta=2$, with $\eps$ replaced by
$\eps^2$ when the $y$ variable is addressed, we have
\begin{align*}
  \|v^\eps(0)\|_{H^{s_1,s_2}(M^2)}
  &\Eq \eps 0
  \frac{1}{\eps}\|\alpha_{1,x}(x)e^{ik_1x/\eps}\|_{H^{s_1}(M)}
    \|\alpha_{1,y}(y)e^{ik_2x/\eps^2}\|_{H^{s_2}(M)}\\
  &\approx
     \eps^{-1-s_1-2s_2},
\end{align*}
so this family goes to zero as $\eps\to 0$ provided that
$s_1+2s_2<-1$. Like  in the KdV case, we can choose initial profiles
$\alpha_{1,x}$ and $\alpha_{1,y} $ so that we have uniformly for
\begin{equation*}
  -K\le \si_1+2\si_2\le s'_1+2s'_2<-1,
\end{equation*}
the lower bound
\begin{equation*}
   \|a_0(\tau)\|_{H^{\si_1,\si_2}(M^2)} >\frac{2}{\delta},
\end{equation*}
where we emphasize the fact that $a_0$ inherits the tensor property in
$(x,y)$ from $\alpha_1$, in view of Proposition~\ref{prop:approx-kp}.
\smallbreak

Invoking the two inequalities,
\begin{itemize}
\item From Proposition~\ref{prop:OG-kp}, 
\begin{equation*}
  \|u^\eps(\tau)-u^\eps_{\rm app}(\tau)\|_{L^2}\le
  \|u^\eps(\tau)-u^\eps_{\rm app}(\tau)\|_{H^3_\eps}\lesssim \eps^2,
\end{equation*}
\item From Lemma~\ref{lem:5.1}, for $s'_1,s'_2\le 0$,
  \begin{equation*}
  \|u^\eps_{\rm app}(\tau)-\eps a_0(\tau)\|_{H^{s'_1,s'_2}}\lesssim
  \eps^{-s'_1-2s'_2},
\end{equation*}
\end{itemize}
we infer, for $\si_1,\si_2\le 0$, with $\si_1+2\si_2\le s'_1+2s'_2$, 
\begin{align*}
  \|v^\eps(\eps^2\tau)-a_0(\tau)\|_{H^{\si_1,\si_2}}
  &\le
  \frac{1}{\eps}\| u^\eps(\tau)-u^\eps_{\rm
    app}(\tau)\|_{H^{\si_1,\si_2}}+  \frac{1}{\eps}\| u^\eps_{\rm
    app}(\tau)-\eps a_0(\tau)\|_{H^{\si_1,\si_2}}\\
  & \le
  \frac{1}{\eps}\| u^\eps(\tau)-u^\eps_{\rm
    app}(\tau)\|_{L^2}+  \frac{1}{\eps}\| u^\eps_{\rm
    app}(\tau)-\eps a_0(\tau)\|_{H^{\si_1,\si_2}}\\
  &\lesssim \eps +
    \eps^{-1-\si_1-2\si_2}\lesssim \eps +
    \eps^{-1-s'_1-2s'_2}\Tend \eps 0 0, 
\end{align*}
hence Theorem~\ref{theo:main-kp}.

\appendix

\section{Action of the nonlinearity on WKB ansatz}
\label{sec:NL}

Suppose that $u^\eps_{\rm app}$ is of the form
\begin{equation}\label{eq:uapp}
  u^\eps_{\rm app} = \sum_{j=1}^3 \eps^{j-1}\(a_j+\eps
  b_j+\eps^2 c_j\) e^{ij\phi^\eps} + \text{c.c.}+ \eps a_0 ,
\end{equation}
where $a_0$ is real-valued. The phase $\phi^\eps$ is given by
\begin{align*}
  \text{KdV: }
  & \phi^\eps(t,x) = \frac{x+t}{\eps},\\
  \text{KP: }
  &  \phi^\eps(t,x,y) =
    \omega\frac{t}{\eps}+k_1\frac{x}{\eps}+k_2\frac{y}{\eps^2}.  
\end{align*}
The property which really matters in this section is the fact that for
all integers $j,k$, $(\eps\d_t)^ke^{ij\phi^\eps}$, $(\eps\d_x)^k
e^{ij\phi^\eps}$ 
and (KP case) $(\eps^2\d_y)^ke^{ij\phi^\eps}$ are bounded uniformly in
$\eps$. 
In the KdV as well as in the KP case, we
want to approximate the exact solution $u^\eps$ up to an error which
is $\O(\eps^2)$ 
in some semiclassical Sobolev space. As $\eps$-derivatives in $t$ or
$x$, as well as $\eps^2$-derivative in $y$ in the KP case, do not
change the size of the above terms in $u^\eps_{\rm app} $, we may
discard the $\O(\eps^3)$ 
contributions in this ansatz, so we assume
$c_2=b_3=c_3=0$.
\smallbreak

We want to order, in terms of harmonics and of powers of $\eps$, the
factors in
\[u^\eps_{\rm app}\eps\d_x u_{\rm
    app}^\eps=\frac{1}{2}\eps\d_x\(u_{\rm app}^\eps\)^2.\]
  We start with
$\(u_{\rm app}^\eps\)^2$. Next, we apply the operator
$\eps\d_x$. For KdV as well as for KP,
there is an extra $\eps$ factor in front of this nonlinear term (weakly
nonlinear geometric optics r\'egime). Also, the
action of $\eps\d_x$ does not amplify WKB terms
like in the expression of $u_{\rm app}^\eps$. 
Therefore,
terms which are $\O(\eps^2)$ in $\(u_{\rm app}^\eps\)^2$ may be
discarded in view of the error estimates we have in mind. In
particular, the terms $c_1$ and  $a_3$ are absent from the
computations below.  
For $u_{\rm app}^\eps$ like in \eqref{eq:uapp}, we order
\begin{align*}
  \(u_{\rm app}^\eps\)^2
  &  = \( \eps a_0 + \sum_{j=1}^3\eps^{j-1}\(a_j+\eps
  b_j+\eps^2 c_j\) e^{ij\phi^\eps} +  \sum_{j=1}^3 \eps^{j-1}\(\bar a_j+\eps
    \bar b_j+\eps^2\bar c_j\) e^{-ij\phi^\eps}\)\\
  &\times \( \eps a_0 + \sum_{k=1}^3 \eps^{k-1}\(a_k+\eps
  b_k+\eps^2 c_k\) e^{ik\phi^\eps} +  \sum_{k=1}^3 \eps^{k-1}\(\bar a_k+\eps
    \bar b_k+\eps^2\bar c_k\) e^{-ik\phi^\eps}\)\\
  &=  e^{i\phi^\eps}\( 2\eps a_0a_1 +2\eps\bar a_1 a_2 \)
   + e^{2i\phi^\eps}\(a_1^2 +2\eps a_1 b_1\)
   + e^{3i\phi^\eps}\(2\eps a_1 a_2 \)
   +\text{c.c.}\\
  &\quad  + 2|a_1|^2 + 2\eps(a_1\bar b_1+\bar a_1 b_1)
  +r_1^\eps,
\end{align*}
with $r_1^\eps=  \O(\eps^2)$ in $H^k_\eps$,
where $H^k_\eps$ is defined by \eqref{eq:H2eps} in the KdV case, by
\eqref{eq:Hkeps} in the KP case, 
and the integer $k$ is arbitrary. For $\phi^\eps$ such that $\eps\d_x
\phi^\eps = k_1\phi^\eps$, we infer
\begin{align*}
 2 \eps u_{\rm app}^\eps \d_x  u_{\rm app}^\eps
   &= ik_1 e^{i\phi^\eps}\( 2\eps a_0a_1 +2\eps\bar a_1 a_2\)
 + e^{i\phi^\eps}\eps^2\d_x\( 2 a_0a_1+2 \bar a_1 a_2 \)\\
  &\quad + 2ik_1e^{2i\phi^\eps}\(a_1^2 +2\eps a_1 b_1\)
 + e^{2i\phi^\eps}\eps\d_x\(a_1^2 +2\eps a_1 b_1\)\\ 
  &\quad + 3ik_1e^{3i\phi^\eps}\(2\eps a_1 a_2 \)
   +e^{3i\phi^\eps}\eps\d_x\(2\eps a_1 a_2 \)
  +\text{c.c.}\\
  &\quad + 2\eps\d_x\(|a_1|^2\) +
    2\eps^2\d_x(a_1\bar b_1+\bar a_1 b_1) +\eps\d_x r_1^\eps.
\end{align*}
We observe that some explicit terms above turn out to be also of order
$\O(\eps^2)$ in $H^k_\eps$, so we adapt the remainder,
\begin{align*}
  r_2^\eps
  &= 2\eps^2 \d_x(a_0a_1) e^{i\phi^\eps}+2\eps^2
    \d_x(a_1b_1)e^{2i\phi^\eps}
    +2\eps^2\d_x(a_1a_2)e^{3i\phi^\eps}+\text{c.c.}\\
  &\quad+  2\eps^2\d_x(a_1\bar b_1+\bar a_1 b_1)  +\eps\d_x r_1^\eps , 
\end{align*}
and the previous expression can be simplified,
\begin{align*}
 2 \eps u_{\rm app}^\eps \d_x  u_{\rm app}^\eps
   &= 2ik_1 \eps \(a_0a_1+\bar a_1 a_2\) e^{i\phi^\eps}
  + \(2ik_1a_1^2 +4ik_1\eps a_1 b_1+ \eps\d_x\(a_1^2 \)\)e^{2i\phi^\eps}\\
  &\quad + 6ik_1\eps a_1a_2 e^{3i\phi^\eps}
   +\text{c.c.}
  + 2\eps\d_x\(|a_1|^2\) +r_2^\eps,
\end{align*}
with $r_2^\eps=  \O(\eps^2) $ in $H^k_\eps$, and where we have
reordered the powers of $\eps$ for each harmonic.  

\section{Construction of the approximate solution: KdV case}
\label{sec:BKW-KdV}

We consider real-valued solutions to \eqref{eq:u-kdv} (or,
equivalently, of \eqref{eq:kdv}): $\alpha_{-1}=\bar\alpha_1$. We seek
for an approximate solution 
to \eqref{eq:u-kdv} of the form
\begin{align*}
  u_{\rm app}^\eps(t,x)
  &= \(a_1(t,x) +\eps b_1(t,x)+\eps^2c_1(t,x)\)
    e^{i\phi_1(t,x)/\eps} +\text{c.c.} \\ 
  &\quad +\eps \( a_2(t,x)+\eps
    b_2(t,x)\)e^{2i\phi_1(t,x)/\eps}+\eps^{2}
    a_3(t,x)e^{3i\phi_1(t,x)/\eps}\\
&\quad +\text{c.c.}+\eps a_0(t,x), 
  \end{align*}
where we recall that $\phi_1$ is given by $\phi_1(t,x)=x+t$. 
Our goal is for $u_{\rm app}^\eps$ to solve \eqref{eq:u-kdv} up to 
$\O(\eps^3)$:
\begin{align*}
  &\(\eps\d_t +\eps^3\d_x^3\)u_{\rm app}^\eps =
    \eps\( (\d_t-3\d_x)a_1 e^{i\phi_1/\eps} -6i a_2
    e^{2i\phi_1/\eps}+\text{c.c.}\)\\
  &\quad
    + \eps^2\(\( 3i\d_x^2 a_1+(\d_t-3\d_x)b_1\)e^{i\phi_1/\eps} +\(
    (\d_t-12\d_x) a_2-6ib_2\)e^{2i\phi_1/\eps}+\text{c.c.} +\d_t
    a_0 \)\\
  &\quad-24 i\eps^2 \(a_3  e^{3i\phi_1/\eps}-\bar a_3
    e^{-3i\phi_1/\eps}\) +\O\(\eps^3\),
 \end{align*}
where $\O(\eps^3)$ is measured in $H^k_\eps$ for (arbitrary) $k\in \N$
(supposing that all the functions considered are smooth). 
On the other hand, computations from Appendix~\ref{sec:NL} (with
$k_1=1$ here) yield
\begin{align*}
  \eps^{2}u_{\rm app}^\eps \d_x u_{\rm app}^\eps
  &= i\eps^2 (a_0a_1+\bar a_1a_2) e^{i\phi_1/\eps} +\(i\eps a_1^2 +2i\eps^2
    a_1b_1+\eps^2a_1\d_x a_1\)e^{2i\phi_1/\eps}\\
  &\quad +
    3i\eps^2 a_1a_2 e^{3i\phi_1/\eps}+\text{c.c.}
    +\eps^2\d_x\(|a_1|^2\)+\O(\eps^3). 
\end{align*}

We cancel the terms in $\eps$ and $\eps^2$, on each multiple of the
$e^{i\phi^\eps}$:
\begin{align*}
  e^{i\phi_1/\eps}: &\quad \O(\eps),\quad  (\d_t-3\d_x)a_1 =0,\\
  & \quad \O(\eps^2),\quad 3i\d_x^2 a_1+(\d_t-3\d_x)b_1 = 6i a_0a_1+6
    i \bar a_1  a_2 ,\\
  e^{2i\phi_1/\eps}: &\quad \O(\eps),\quad -6ia_2 = 6 i a_1^2,\\
  & \quad \O(\eps^2),\quad (\d_t-12\d_x) a_2-6ib_2=12i a_1b_1+ 6 a_1
    \d_x a_1 ,\\
 e^{3i\phi_1/\eps}: &\quad \O(\eps^2),\quad -24ia_3 = 18ia_1a_2,\\
 e^{0}: &\quad \O(\eps^2),\quad \d_t a_0=6 \d_x |a_1|^2.
\end{align*}
Imposing $a_{1\mid t=0}=\alpha_1$, we find
\begin{equation*}
  a_1(t,x) = \alpha_1(x+3t),\quad a_2(t,x) = -
  \alpha_1(x+3t)^2, 
\end{equation*}
hence in particular $a_{2\mid t=0}= -
  \alpha_1^2$. This shows that we actually consider initial data for
  $u^\eps$ which
  are not exactly like discussed so far, but well-prepared in the sense
  that this $\O(\eps)$ correction is fixed at time $t=0$. The discussion
 will be different in the case of $a_3$, because of its factor
 $\eps^2$, since we eventually want to show that $u^\eps-u^\eps_{\rm
   app}=\O(\eps^2)$ in $H^2_\eps$.

 We first compute $a_0$,
 \begin{equation*}
    a_0(t,x) = 6\int_0^t \d_x|a_1(s,x)|^2ds =
    2\int_0^t \d_s|a_1(s,x)|^2ds = 2\(
    |\alpha_1(x+3t)|^2-|\alpha_1(x)|^2\),  
  \end{equation*}
 and then infer $b_1$:
 \begin{equation}\label{eq:b1}
   \begin{aligned}
    b_1(t,x)
    &=-i \int_0^t
    \(\bar a_1 a_2+ a_0a_1 +3\d_x^2
      a_1\)(s,x+3(t-s))ds.
      \end{aligned}
  \end{equation}
 Like in the case of  $a_2$, $b_2$ and $a_3$  are given by formulas
which show that their initial data are not zero. 

\section{Construction of the approximate solution: KP case}
\label{sec:BKW-KP}

To guarantee that the source term in \eqref{eq:uapp-KP} is an
$x$-derivative, we seek an approximate solution of the form
\begin{equation}\label{eq:forme-uapp-KP}
  u^\eps_{\rm app} = \eps\d_x\(\sum_{j=1}^3 \eps^{j-1}\(a_j+\eps
  b_j+\eps^2 c_j\) e^{ij\phi^\eps} + \text{c.c.}\)+ \eps a_0 ,
\end{equation}
which requires implicitly (again) that the initial data for $u^\eps$
are well-prepared, in a sense that we make precise later. In the KP
case, the phase $\phi^\eps$ is given by \eqref{eq:eik-KP}. Up to
changing notations, this form is the same as in \eqref{eq:uapp}, since
\begin{equation*}
  \eps\d_x\( f e^{ij\phi^\eps}\) = ijk_1 f e^{ij\phi^\eps} + (\eps
  \d_x f) e^{ij\phi^\eps}.
\end{equation*}
Setting $\tilde a_0 = a_0$, and, for $j\ge 1$,
\begin{equation}\label{eq:tilde}
   \tilde a_j =ijk_1 a_j,\quad
   \tilde b_j = ijk_1b_j + \d_x a_j,\quad
   \tilde c_j = ijk_1 c_j + \d_x b_j,
\end{equation}
$u_{\rm app}^\eps$ is \eqref{eq:forme-uapp-KP} has the same form as in
\eqref{eq:uapp}, provided tildas are added, and neglecting the
$\O(\eps^3)$ terms $\eps^3 \d_x c_j$. Again, we may assume
$\tilde c_2=\tilde b_3=\tilde c_3=0$, but we may also choose to
incorporate such contributions into the remainder term. 
\smallbreak

In agreement with the above notations, we compute successively
\begin{align*}
  \eps \d_t u^\eps_{\rm app}
  &=\sum_{j=1}^3 \eps^{j-1}ij\omega
   \(\tilde a_j+\eps
  \tilde b_j+\eps^2 \tilde c_j\) e^{ij\phi^\eps} + \sum_{j=1}^3 \eps^{j-1}
   \eps\d_t\(\tilde a_j+\eps
    \tilde b_j+\eps^2 \tilde c_j\) e^{ij\phi^\eps}\\
  &\quad +\text{c.c.}+\eps^2\d_t  a_0 ,
\end{align*}
\begin{align*}
  \eps^2\d_x \d_t u^\eps_{\rm app}
  &=\sum_{j=1}^3 \eps^{j-1}(ij\omega)(ijk_1)
   \(\tilde a_j+\eps
    \tilde b_j+\eps^2 \tilde c_j\) e^{ij\phi^\eps} \\
  &\quad
  + \sum_{j=1}^3 \eps^{j-1}ij\omega
  \eps\d_x \(\tilde a_j+\eps
    \tilde b_j+\eps^2 \tilde c_j\) e^{ij\phi^\eps}  \\
  &\quad   +
    \sum_{j=1}^3 \eps^{j-1}
   ijk_1\eps\d_t\(\tilde a_j+\eps
    \tilde b_j+\eps^2 \tilde c_j\) e^{ij\phi^\eps}\\
  &\quad +
    \sum_{j=1}^3 \eps^{j-1}
   \eps^2\d_x\d_t \(\tilde a_j+\eps
    \tilde b_j+\eps^2 \tilde c_j\) e^{ij\phi^\eps}
  +\text{c.c.}+\eps^3\d_x\d_t  a_0 ,
\end{align*}
\begin{align*}
  \eps^4 \d_x^4 u^\eps_{\rm app} 
  &= \sum_{j=1}^3 \eps^{j-1}(jk_1)^4 \(\tilde a_j+\eps
    \tilde b_j+\eps^2 \tilde c_j\) e^{ij\phi^\eps} \\
  &\quad-4i \sum_{j=1}^3 \eps^{j-1}(jk_1)^3 \eps\d_x\(\tilde a_j+\eps
    \tilde b_j+\eps^2 \tilde c_j\) e^{ij\phi^\eps}\\
 &\quad-4\sum_{j=1}^3 \eps^{j-1}(jk_1)^2 (\eps\d_x)^2\(\tilde a_j+\eps
    \tilde b_j+\eps^2 \tilde c_j\) e^{ij\phi^\eps}\\
 &\quad+4i \sum_{j=1}^3 \eps^{j-1}jk_1 (\eps\d_x)^3\(\tilde a_j+\eps
    \tilde b_j+\eps^2 \tilde c_j\) e^{ij\phi^\eps}\\
 &\quad+ \sum_{j=1}^3 \eps^{j-1}(\eps\d_x)^4 \(\tilde a_j+\eps
    \tilde b_j+\eps^2 \tilde c_j\) e^{ij\phi^\eps}\\
 &\quad+ \sum_{j=1}^3 \eps^{j-1}(jk_1)^4 \(\tilde a_j+\eps
   \tilde b_j+\eps^2 \tilde c_j\) e^{ij\phi^\eps}
    +\text{c.c.}+\eps^5\d_x^4  a_0 ,
\end{align*}
\begin{align*}
  \eps^4 \d_y^2 u^\eps_{\rm app} 
  &= - \sum_{j=1}^3 \eps^{j-1}(jk_2)^2 \(\tilde a_j+\eps
    \tilde b_j+\eps^2 \tilde c_j\) e^{ij\phi^\eps} \\
  &\quad+2i \sum_{j=1}^3 \eps^{j-1}jk_2 (\eps^2\d_y)\(\tilde a_j+\eps
    \tilde b_j+\eps^2 \tilde c_j\) e^{ij\phi^\eps}\\
 &\quad +\sum_{j=1}^3 \eps^{j-1}(\eps^4\d_y^2)\(\tilde a_j+\eps
   \tilde b_j+\eps^2 \tilde c_j\) e^{ij\phi^\eps}
 +\text{c.c.}+\eps^5\d_y^2  a_0 .  
\end{align*}
Plugging these expressions into \eqref{eq:uapp-KP}, and using the
computations from Appendix~\ref{sec:NL}, we find that the only term of order
$\O(1)$ is given by:
\begin{equation*}
  \O(1): \quad \(-\omega k_1 +k_1^4-\lambda k_2^2 \)a_1 e^{i\phi^\eps}. 
\end{equation*}
Since we want to consider a nontrivial leading order amplitude,
$a_1\not\equiv 0$, we impose that the first factor is zero. For
$k_1\not =0$, we find
\begin{equation}\label{eq:disp-KP}
  \omega = k_1^3-\lambda\frac{k_2^2}{k_1}. 
\end{equation}
For other terms in the WKB hierarchy, we recall that we aim at writing
\eqref{eq:uapp-KP}, with $\Sigma^\eps=\O(\eps^3)$ in $H^3_\eps$. At
this stage, we factor out the semiclassical operator $\eps\d_x$. To do
so, we recall that we have the two relations
\begin{align*}
  u^\eps_{\rm app}
  &= \eps\d_x\(\sum_{j=1}^3 \eps^{j-1}\(a_j+\eps
    b_j+\eps^2 c_j\) e^{ij\phi^\eps} + \text{c.c.}\)+ \eps a_0 \\
  &= \sum_{j=1}^3 \eps^{j-1}\(\tilde a_j+\eps
   \tilde b_j+\eps^2 \tilde c_j+\eps^3 \d_x c_j\) e^{ij\phi^\eps} +
    \text{c.c.}+ \eps a_0  ,
\end{align*}
so we get
\begin{align*}
  \Sigma^\eps
  &=\sum_{j=1}^3 \eps^{j-1}ij\omega
   \( \tilde a_j+\eps
    \tilde b_j+\eps^2   \tilde c_j\) e^{ij\phi^\eps} + \sum_{j=1}^3 \eps^{j-1}
   \eps\d_t\(  \tilde a_j+\eps
      \tilde b_j+\eps^2  \tilde c_j\) e^{ij\phi^\eps}\\
&\quad-i  \sum_{j=1}^3 \eps^{j-1}(jk_1)^3 \(\tilde a_j+\eps
    \tilde b_j+\eps^2 \tilde c_j\) e^{ij\phi^\eps} \\
  &\quad-3 \sum_{j=1}^3 \eps^{j-1}(jk_1)^2 \eps\d_x\(\tilde a_j+\eps
    \tilde b_j+\eps^2 \tilde c_j\) e^{ij\phi^\eps}\\
 &\quad+3i\sum_{j=1}^3 \eps^{j-1}jk_1 (\eps\d_x)^2\(\tilde a_j+\eps
    \tilde b_j+\eps^2 \tilde c_j\) e^{ij\phi^\eps}\\
 &\quad+ \sum_{j=1}^3 \eps^{j-1} (\eps\d_x)^3\(\tilde a_j+\eps
    \tilde b_j+\eps^2 \tilde c_j\) e^{ij\phi^\eps}\\
 & \quad -\lambda  \sum_{j=1}^3 \eps^{j-1}(jk_2)^2 \( a_j+\eps
     b_j+\eps^2 c_j\) e^{ij\phi^\eps} \\
  &\quad+2i\lambda \sum_{j=1}^3 \eps^{j-1}jk_2 (\eps^2\d_y)\( a_j+\eps
     b_j+\eps^2  c_j\) e^{ij\phi^\eps}\\
 &\quad +\lambda \sum_{j=1}^3 \eps^{j-1}(\eps^4\d_y^2)\( a_j+\eps
    b_j+\eps^2  c_j\) e^{ij\phi^\eps}\\
  &\quad +\text{c.c.}+\eps^2\d_t  a_0+\eps^4\d_x^3  a_0+\eps^4
    \d_x^{-1}\d_y^2  a_0  
  + \frac{\eps}{2}(\eps\d_x)\(u_{\rm app}^2\).
\end{align*}
We now require algebraic relations leading to the property
$\Sigma^\eps=\O(\eps^3)$ in $H^3_\eps$, taking the computations from
Appendix~\ref{sec:NL} into account, by ordering like in the KdV case:
for the first harmonic,
\begin{align*}
  e^{i\phi^\eps}:
 &\quad \O(\eps), \quad \d_t \tilde a_1+i\omega \tilde  b_1 -ik_1^3
   \tilde b_1-3 k_1^2 \d_x \tilde a_1-\lambda k_2^2 b_1=0.\\
&\quad \O(\eps^2), \quad  \d_t \tilde b_1 + i\omega \tilde c_1 -ik_1^3
  \tilde c_1 -3k_1^2\d_x \tilde b_1 +3i\d_x^2\tilde a_1-\lambda k_2^2 c_1
  +2i\lambda k_2\d_y a_1\\
&\quad\phantom{\O(\eps^2), } \qquad +ik_1\(a_0\tilde a_1 +\overline{\tilde
  a_1}\tilde a_2\)=0.
\end{align*}
In view of \eqref{eq:tilde}, for $k_1\not =0$,
\begin{equation*}
  b_1 = \frac{1}{ik_1}\tilde b_1 +\frac{1}{k_1^2}\d_x \tilde a_1,
\end{equation*}
so the first equation becomes, in view of the dispersion relation
\eqref{eq:disp-KP},
\begin{equation*}
  \d_t \tilde a_1-3 k_1^2 \d_x \tilde
  a_1-\lambda\frac{k_2^2}{k_1^2}\d_x \tilde a_1=0. 
\end{equation*}
We therefore have the explicit formula
\begin{equation*}
  \tilde a_1(t,x,y) = \tilde
  a_1\(0,x+\(3k_1^2+\lambda\frac{k_2^2}{k_1^2}\)t,y\) =  ik_1
  \alpha_1\(x+\(3k_1^2+\lambda\frac{k_2^2}{k_1^2}\)t,y\). 
\end{equation*}
Similarly, the second equation becomes
\begin{equation*}
  \d_t \tilde b_1-3 k_1^2 \d_x \tilde
  b_1-\lambda\frac{k_2^2}{k_1^2}\d_x \tilde b_1+3i\d_x^2\tilde
  a_1+2\lambda\frac{k_2}{k_1}\d_y \tilde a_1+\lambda
\frac{k_2^2}{ik_1^3}\d_x^2\tilde a_1 +ik_1\(a_0\tilde a_1 +\overline{\tilde
  a_1}\tilde a_2\)=0. 
\end{equation*}
For the other harmonics, the dispersion relation does not imply (the
same) cancellations, and we find:
\begin{align*}
  e^{2i\phi^\eps}:
 &\quad \O(\eps), \quad 2i\omega\tilde a_2 -i(2k_1)^3 \tilde
   a_2-\lambda (2k_2)^2 a_2+ik_1 (\tilde a_1)^2 .\\
&\quad \O(\eps^2), \quad  2i\omega \tilde b_2 +\d_t \tilde a_2
  -i(2k_1)^3 \tilde b_2-3(2k_1)^2\d_x \tilde a_2-\lambda (2k_2)^2
  b_2\\
&\quad\phantom{ \O(\eps^2),}  \qquad+2ik_1\tilde a_1\tilde b_1 +\tilde a_1\d_x\tilde a_1=0,\\
  e^{3i\phi^\eps}:
&\quad \O(\eps^2), \quad  3i\omega \tilde a_3-i(3k_1)^3 \tilde
  a_3-\lambda (3k_2)^2a_3 +3ik_1\tilde a_1\tilde a_2=0,\\
  e^{0}:
&\quad \O(\eps^2), \quad  \d_t a_0 +\d_x\(|\tilde a_1|^2\)=0.
\end{align*}
At this stage, the discussion mimics the one we had in the KdV case to
construct the approximate solution.

\bibliographystyle{abbrv}
\bibliography{inflation}

\end{document}